\documentclass[12pt,amsymb,fullpage]{amsart}
\usepackage{amssymb,amscd,pstricks}

\newtheorem{theorem}{Theorem}[section]
\newtheorem{defn}[theorem]{Definition}

\newtheorem{lemma}[theorem]{Lemma}

\newtheorem{eple}[theorem]{Example}
\newtheorem{rmk}[theorem]{Remarks}
\newtheorem{dsc}[theorem]{Discussion}
\newtheorem{nota}[theorem]{Notation}

\newsavebox{\indbin}
\savebox{\indbin}{\begin{picture}(0,0)
\newlength{\gnu}
\settowidth{\gnu}{$\smile$} \setlength{\unitlength}{.5\gnu}
\put(-1,-.65){$\smile$} \put(-.25,.1){$|$}
\end{picture}}

\newcommand{\be}{\begin{enumerate}}
\newcommand{\bd}{\begin{defn}}
\newcommand{\bt}{\begin{theorem}}
\newcommand{\bl}{\begin{lemma}}
\newcommand{\ee}{\end{enumerate}}
\newcommand{\ed}{\end{defn}}
\newcommand{\et}{\end{theorem}}
\newcommand{\el}{\end{lemma}}

\begin{document}
\title{A Nonstandard Approach to Equidistribution}
\author{Tristram de Piro}
\address{Mathematics Department, Harrison Building, Streatham Campus, University of Exeter, North Park Road, Exeter, Devon, EX4 4QF, United Kingdom}
\begin{abstract}
Using nonstandard analysis, we generalise a classical result on equidistributions to integrable functions, and give an application of the Weil conjectures for algebraic curves, to equidistribution in characteristic zero.
\end{abstract}
\maketitle

\begin{lemma}
\label{equidistributed1}

Let $\{x_{n}:n\in\mathcal{N}\}$ be equidistributed on $[0,1)$, then, if $f\in L^{1}([0,1))$, we have that;\\

$\int_{0}^{1}f d\mu=lim_{n\rightarrow \infty}{1\over n}\sum_{j=1}^{n}f(x_{j})$\\

\end{lemma}

\begin{proof}
We give a nonstandard proof of this result. Choose $\eta\in{^{*}\mathcal N}$ infinite. By transfer, we can find an internal sequence $\{s_{i}:i\in ({^{*}\mathcal N}\cap [1,\eta])\}\subset {^{*}[0,1)}$, with $s_{i}=x_{i}$, for $i\in\mathcal{N}$. We adopt the notation $(\overline{S}_{\eta},{\mathcal{C}}_{\eta})$ of Definition 0.4 in \cite{dep3}, and define $\delta_{\eta}:{\mathcal{C}}_{\eta}\rightarrow {^{*}\mathcal{R}}$ by setting;\\

$\delta_{\eta}([{j\over\eta},{j+1\over \eta}))={1\over\eta}{^{*}{Card}}(\{i\in {^{*}\mathcal N}\cap [1,\eta]:s_{i}\in {[{j\over\eta},{j+1\over \eta})}\})$\\

$\delta_{\eta}(V)={^{*}\sum}_{{j\over \eta}\in V}\delta_{\eta}([{j\over\eta},{j+1\over \eta}))$ $(*)$\\

for $0\leq j\leq \eta-1$ and $V\in {\mathcal{C}}_{\eta}$. It is easily verified that $\delta_{\eta}$ is finitely additive, hence $*$-finitely additive. Following \cite{dep4}, we let $(L(\overline{S}_{\eta}),L({\mathcal{C}}_{\eta}),L(\delta_{\eta}))$ denote the associated Loeb space. We claim that the standard part mapping;\\

$st:(L(\overline{S}_{\eta}),L({\mathcal{C}}_{\eta}),L(\delta_{\eta}))\rightarrow ([0,1],\mathcal{B},\mu)$\\

is measurable and measure preserving, $(**)$, where $\mathcal{B}$ denotes the completion of the Borel field on $[0,1]$ and $\mu$ is Lebesgue measure. Observe, for $\{a,b\}\subset ([0,1)\cap \mathcal{Q})$, using $(*)$ and the fact that $[a,b)=\bigcup_{a\eta\leq j<b\eta}[{j\over\eta},{j+1\over \eta})$;\\

$\delta_{\eta}({^{*}[a,b)})={1\over\eta}{^{*}{Card}}\{1\leq i\leq\eta:s_{i}\in {[a,b)}\}$\\

The internal sequence $\{s_{a,b}^{i}\}_{1\leq i\leq \eta}$, defined by $s_{a,b}^{i}={1\over i}{^{*}Card}(\{k\in {^{*}\mathcal{N}}\cap [1,i]:s_{k}\in {^{*}[a,b)}\})$, has the property that $s_{a,b}^{\eta}\simeq b-a$, using Theorem 2.22(i) of \cite{dep4}. Hence $L(\delta_{\eta})({^{*}[a,b)})=b-a$. Now, let $\{c,d\}\subset ([0,1)\cap \mathcal{R})$, and assume that $c\neq 0$, (\footnote{The case $c=0$ can be dealt with, by observing that $[\eta 0]=0$, and taking $c_{u,n}=0$.}). Choose sequences $\{c_{l,n},c_{u,n},d_{l,n},d_{u,n}:n\in\mathcal{Z}_{\geq 1}\}\subset ({^{*}[0,1)}\cap \mathcal{Q})$ such that $c_{u,n}<c<c_{l,n}<d_{l,n}<d<d_{u,n}$, $lim_{n\rightarrow \infty}c_{u,n}=lim_{n\rightarrow \infty}c_{l,n}=c$ and $lim_{n\rightarrow \infty}d_{u,n}=lim_{n\rightarrow \infty}d_{l,n}=d$. We have that $[c_{l,n},d_{l,n})\subset [{[\eta c]\over \eta},{[\eta d]\over \eta})\subset [c_{u,n},d_{u,n})$, for $n\in\mathcal{Z}_{\geq 1}$. Then, using elementary properties of measures, we have that;\\

 $L(\delta_{\eta})([{[\eta c]\over \eta},{[\eta d]\over \eta}))=lim_{n\rightarrow \infty}L(\delta_{\eta})([c_{l,n},d_{l,n}))$\\

 $=lim_{n\rightarrow \infty}L(\delta_{\eta})([c_{u,n},d_{u,n}))$\\

 $=lim_{n\rightarrow \infty}(d_{l,n}-c_{l,n})$\\

 $=lim_{n\rightarrow \infty}(d_{u,n}-c_{u,n})=d-c$\\

 We can now follow Theorem 14 in \cite{A}, to obtain that $L(\delta_{\eta})(st^{-1}([c,d)))=d-c$, and then $(**)$ is shown, using the same proof. For $g\in V(\overline{S}_{\eta})$, and $A\in {\mathcal{C}}_{\eta}$, we let $\int_{A} g d\delta_{\eta}$ be as in Definition 3.9 of \cite{dep4}, and define $S$-integrability, as in Definition 3.17 of \cite{dep4}. Then, we have, by Theorem 3.20 of \cite{dep4}, that, for $g$ $S$-integrable;\\

${^{\circ}\int_{\overline{S}_{\eta}}g d\delta_{\eta}}=\int_{\overline{S}_{\eta}}{^{\circ}g}dL(\delta_{\eta})$, $(***)$\\

If $f\in L^{1}([0,1),\mathcal{B},\mu)$, using the result $(**)$, we must have that $st^{*}(f)\in L^{1}(L(\overline{S}_{\eta}),L({\mathcal{C}}_{\eta}),L(\delta_{\eta}))$. We claim that there exists $g\in SL^{1}(\overline{S}_{\eta})$, (\footnote{Using the notation in \cite{A} for $S$-integrable functions.}), with the property that $g(x_{i})=f(x_{i})$, for $1\leq i\leq \eta$ and ${^{\circ}g}=st^{*}(f)$ a.e $d(L(\delta_{\eta}))$, $(****)$. We follow the case by case proof of Theorem 3.31 in \cite{dep4}. The case when $st^{*}(f)$ is bounded follows by choosing the initial sequence of $\mathcal{C}_{\eta}$-measurable functions $\{f_{n}\}_{n\in\mathcal{N}_{>0}}$ to have the property that $f_{n}(x_{i})=f(x_{i})$, for $1\leq i\leq n$. After extending the sequence $\{f_{n}\}_{n\in\mathcal{N}_{>0}}$ to an internal sequence $\{f_{n}\}_{1\leq n\leq \omega'}$, for some infinite $\omega'$, this property  continues to hold by overflow, quantifying over the internal sequence $\{{^{*}f}(s_{i})\}_{1\leq i\leq min(\omega',\eta)}$. Choosing $\omega\leq \omega'$, as in the proof of Theorem 3.13, we obtain that $f_{\omega}(x_{i})=f(x_{i})$, for $i\in\mathcal{N}$, $(*****)$. For the general case, we can follow the proof, requiring, using $(*****)$, and replacing $\overline{S}_{\eta}$ by $A_{n}$, that the sequence $\{f_{n}\}_{n\in\mathcal{N}_{>0}}$, has the property that $f_{n}(x_{i})=f(x_{i})$, for any $s_{i}\in A_{n}$. Hence, $(****)$ is shown. Then, using $(**),(***),(****)$;\\

${^{\circ}({1\over \eta}{^{*}\sum_{1\leq j\leq \eta}f(s_{j})})}$\\

$={^{\circ}({1\over \eta}{^{*}\sum_{1\leq j\leq \eta}g(s_{j})})}$\\

$={^{\circ}\int_{\overline{S}_{\eta}}g d\delta_{\eta}}$\\

$=\int_{\overline{S}_{\eta}}{^{\circ}g}dL(\delta_{\eta})$\\

$=\int_{\overline{S}_{\eta}}st^{*}(f)dL(\delta_{\eta})=\int_{0}^{1}fd\mu$\\

The lemma then follows, this time using Theorem 2.22(ii) of \cite{dep4}.

\end{proof}

\begin{defn}
\label{equidistributed2}
If $\eta\in{^{*}\mathcal{N}}$ is infinite, we say that an internal sequence $\{s_{i}\}_{1\leq i\leq \eta}\subset {^{*}[0,1)}$ is equidistributed if it corresponds, by transfer, to a standard equidistributed sequence $\{x_{i}\}_{i\in\mathcal{Z}_{\geq 1}}\subset [0,1)$. An internal sequence $\{s_{i}\}_{1\leq i\leq \eta}$ is weakly equidistributed if, for the associated measure $L(\delta_{\eta})$, $L(\delta_{\eta})(a,b)=b-a$, for $\{a,b\}\subset {^{*}[0,1)}$.
\end{defn}

\begin{rmk}
\label{evaluate}
Observe, from the proof of Lemma \ref{equidistributed1},  that equidistributed implies weakly equidistributed, and, if $\{s_{i}\}_{1\leq i\leq \eta}$ is equidistributed or weakly equidistributed, then for any standard $f\in L^{1}([0,1))$, $({1\over \eta}{^{*}\sum_{1\leq j\leq \eta}f(s_{j})})\simeq \int_{0}^{1}f d\mu$.
\end{rmk}

\begin{lemma}
\label{weylscriterion}
If $\eta\in{^{*}\mathcal{N}}$ is infinite,an internal sequence $\{s_{i}\}_{1\leq i\leq \eta}$ is weakly equidistributed iff ${1\over\eta}{^{*}\sum}_{1\leq i\leq \eta}exp_{\eta}(2\pi iks_{i})\simeq 0$, for finite $k\in\overline{\mathcal{Z}}_{\eta,\neq 0}$, (\footnote{We adopt the notation of Definition 0.8 in \cite{dep3}, letting $exp_{\eta}(2\pi ikx)$ denote the ${\mathcal{C}}_{\eta}$-measurable counterpart of ${^{*}exp}(2\pi ikx)$ on ${^{*}[0,1)}$, and $\overline{\mathcal{Z}}_{\eta,\neq 0}=\{k\in{^{*}\mathcal{Z}}:-\eta\leq k\leq \eta-1\}$.}).
\end{lemma}

\begin{proof}
Suppose that $\{s_{i}\}_{1\leq i\leq \eta}$ is weakly equidistributed, then, using the proof of Lemma \ref{equidistributed1} and Remark \ref{equidistributed2}, we have that, for finite $k\in \mathcal{Z}_{\eta,\neq 0}$, as $exp_{\eta}(2\pi ikx)$ is $S$-integrable;\\

${^{\circ}({1\over \eta}{^{*}\sum_{1\leq j\leq \eta}}exp_{\eta}(2\pi iks_{j}))}$\\

$={^{\circ}\int_{\overline{S}_{\eta}}exp_{\eta}(2\pi ikx) d\delta_{\eta}}$\\

$=\int_{\overline{S}_{\eta}}{^{\circ}exp_{\eta}(2\pi ikx)}dL(\delta_{\eta})$\\

$=\int_{\overline{S}_{\eta}}st^{*}(exp(2\pi ikx))dL(\delta_{\eta})=\int_{0}^{1}exp(2\pi ikx)d\mu=0$\\

Conversely, suppose that ${1\over\eta}{^{*}\sum}_{1\leq i\leq \eta}exp_{\eta}(2\pi iks_{i})\simeq 0$, $(*)$, for finite $k\in\mathcal{Z}_{\eta,\neq 0}$. Let $\{a,b\}\subset [0,1)$, $\epsilon>0$ and choose $f\in C^{\infty}([0,1])$, (\footnote{We let $C^{\infty}([0,1])=\{f\in C[0,1]:\exists g\in C^{\infty}(S^{1}), ang^{*}g=f\}$, where $ang(\theta)=e^{2\pi i\theta}$, for $\theta\in [0,1]$.}), such that $||f-\chi_{[a,b]}||_{C([0,1))}<\epsilon$, $(**)$. Suppose that $f=g+r$, where $r=\int_{0}^{1}fd\mu$, so that $g\in C^{\infty}([0,1])$ and $\int_{0}^{1}g d\mu=0$. Using Lemma 0.9 of \cite{dep3}, we have that;\\

$g_{\eta}(x)={^{*}\sum}_{k\in\mathcal{Z}_{\eta,\neq 0}}\hat{g_{\eta}}(k)exp_{\eta}(2\pi ikx)$, (\footnote{We adopt the notation, in Definition 0.8, for $\{g_{\eta},{\hat{g}}_{\eta},{\overline{Z}}_{\eta}\}$.})\\

Hence, using $(*)$, the fact that $|{1\over\eta}{^{*}\sum}_{1\leq i\leq \eta}exp_{\eta}(2\pi iks_{i})|\leq 1$, and $|\hat{g_{\eta}}(k)|\leq {H\over k^{2}}$, $(***)$, for $k\in{\overline{\mathcal{Z}}}_{\eta}$, where $H\in\mathcal{R}$, (\footnote{For $(***)$, see Lemma 0.16 and Theorem 0.19 of \cite{dep3}.});\\

${1\over\eta}{^{*}\sum}_{1\leq i\leq \eta}g_{\eta}(s_{i})={1\over\eta}{^{*}\sum}_{1\leq i\leq \eta}{^{*}\sum}_{k\in\mathcal{Z}_{\eta,\neq 0}}\hat{g_{\eta}}(k)exp_{\eta}(2\pi iks_{i})$\\

$={^{*}\sum}_{k\in\mathcal{Z}_{\eta,\neq 0}}\hat{g_{\eta}}(k){1\over\eta}{^{*}\sum}_{1\leq i\leq \eta}exp_{\eta}(2\pi iks_{i})\simeq 0$\\

Hence;\\

${1\over\eta}{^{*}\sum}_{1\leq i\leq \eta}f_{\eta}(s_{i})\simeq r$\\

and, using $(**)$;\\

$|{1\over\eta}{^{*}\sum}_{1\leq i\leq \eta}\chi_{[a,b),\eta}(s_{i})-{1\over\eta}{^{*}\sum}_{1\leq i\leq \eta}f_{\eta}(s_{i})|\leq {1\over \eta}\eta\epsilon=\epsilon$\\

Hence, as $|{1\over\eta}{^{*}\sum}_{1\leq i\leq \eta}\chi_{(a,b),\eta}(s_{i})-r|<2\epsilon$ and $|r-(b-a)|<\epsilon$, we have that;\\

$|{1\over\eta}{^{*}\sum}_{1\leq i\leq \eta}\chi_{(a,b),\eta}(s_{i})-(b-a)|<3\epsilon$\\

and, as $\epsilon$ was arbitrary;\\

${1\over\eta}{^{*}\sum}_{1\leq i\leq \eta}\chi_{(a,b),\eta}(s_{i})\simeq (b-a)$\\

It follows that $\{s_{i}\}_{1\leq i\leq \eta}$ is weakly equidistributed.

\end{proof}

\begin{lemma}
\label{polynomials}
Let $p\in \mathcal{R}[x]$ be a standard polynomial of degree $d$, $p(x)=\sum_{l=0}^{d}a_{l}x^{l}$, with $0\leq a_{l}<1$, then;\\

$lim_{q\rightarrow\infty,q\ prime}{[a_{l}q]\over q}=a_{l}$\\

and, if $0\leq a<b<1$ and $f\in L^{1}([0,1))$;\\

$lim_{q\rightarrow\infty,q\ prime}{1\over q}|\{i:1\leq i\leq q, \sum_{l=0}^{d}{[a_{l}q]\over q}i^{l}\ (mod\ 1)\in (a,b)\}|=(b-a)$\\

$\int_{0}^{1}f d\mu=lim_{q\rightarrow\infty,q\ prime}{1\over q}\sum_{i=1}^{q}f(\sum_{l=0}^{d}{[a_{l}q]\over q}i^{l}\ (mod\ 1))$\\

\end{lemma}
\begin{proof}
The first claim follows easily from the fact that, for infinite $\eta\in{^{*}\mathcal{N}}$ prime, ${[a_{l}\eta]\over \eta}\simeq a_{l}$, and Theorem 2.2(i) of \cite{dep3}. For $q$ prime in  ${^{*}\mathcal{N}}$, let $t_{l,q}=[q a_{l}]$. We have $a_{l}\simeq {[\eta a_{l}]\over \eta}$, and, therefore, $0\leq t_{l,\eta}<\eta$, for $\eta\in{^{*}\mathcal{N}}$ infinite prime. It follows, using underflow, that $0\leq t_{l,q}<q$, for sufficiently large $q\in\mathcal{N}$ prime, $q\geq N(p)$, $(*)$. For $q\in {^{*}\mathcal{N}}$ prime, let $p_{q}=\sum_{l=0}^{d}t_{l,q}x^{l}$. We now claim that, for infinite $\eta\in{^{*}\mathcal{N}}$ prime, the sequence $\{{^{*}p_{\eta}\over\eta}(j)\}_{1\leq j\leq \eta}$ is weakly equidistributed, $(**)$. By Lemma \ref{weylscriterion}, it is sufficient to show that there exists an infinite $\eta\in{^{*}\mathcal{N}}$, with ${1\over\eta}{^{*}\sum}_{1\leq j\leq \eta}exp_{\eta}(2\pi ik{p_{\eta}\over \eta}(j))\simeq 0$, for $k\in\overline{{\mathcal{Z}}}_{\eta}$, $k$ finite, $(***)$.\\

Let $F_{q}\cong \mathcal{Z}/ q\mathcal{Z}$ denote a finite field with $q$ elements. Using Lemma 0.5 of \cite{dep5}, we have that, for $q\geq N(p)$, $(q,d)=1$, and for $0<k\leq \eta-1$;\\

$|\sum_{0\leq j\leq q-1}e^{2\pi i {k\over q}p_{q}(j)}|\leq (d-1)q^{{1\over 2}}+1$\\

If $\eta\in{^{*}\mathcal{N}}$ is prime, then $(\eta,d)=1$, $\eta\geq N(p)$, and, by transfer, for $0<k\leq \eta-1$;\\

${1\over \eta}|\sum_{0\leq j\leq \eta-1}{^{*}exp}(2\pi i {k\over \eta}{^{*}p_{\eta}}(j))|\leq {(d-1)\over \eta^{1\over 2}}+{1\over \eta}\simeq 0$\\

The characters $\{e^{2\pi i {k\over q}}:-(q-1)\leq k\leq -1\}$ are just a re-enumeration of the characters $\{e^{2\pi i {k\over q}}:1\leq k\leq q-1\}$ on $F_{q}$ for $q$ prime, and, therefore, by the same argument;\\

${1\over \eta}|\sum_{0\leq j\leq \eta-1}{^{*}exp}(2\pi i {k\over \eta}{^{*}p_{\eta}}(j))|\simeq 0$, for $k\in{\mathcal{Z}_{\eta}\setminus \{-\eta,0\}}$\\

As $exp(2\pi ikx)$ is continuous on $[0,1)$, for $k\in\mathcal{Z}$, and;\\
 
$max({1\over \eta}exp_{\eta}(2\pi i {k\over \eta}{^{*}p_{\eta}}(0)),{1\over\eta}(exp_{\eta}2\pi i {k\over \eta}{^{*}p_{\eta}}(\eta)))\simeq 0$\\

we have that;\\

${1\over \eta}|{^{*}\sum}_{1\leq j\leq \eta}exp_{\eta}(2\pi i {k\over \eta}{^{*}p_{\eta}}(j))|\simeq 0$, for finite $k\in\mathcal{Z}_{\eta,\neq 0}$\\

It follows that $(**),(***)$ hold. We have that, for any given $\epsilon>0$ standard, $\eta\in{^{*}\mathcal{N}}$ infinite prime, that;\\

$(b-a)-\epsilon<{1\over\eta}|\{i:1\leq i\leq \eta,\sum_{l=0}^{d}{[a_{l}\eta]\over \eta}i^{l}\ (mod\ 1)\in (a,b)\}|<(b-a)+\epsilon$\\

By underflow, there exists a standard $N(\epsilon,p)$ prime, such that, for all standard primes $q\geq N(\epsilon,p)$;\\

$(b-a)-\epsilon<{1\over q}|\{i:1\leq i\leq q,\sum_{l=0}^{d}{[a_{l}q]\over q}i^{l}\ (mod\ 1)\in (a,b)\}|<(b-a)+\epsilon$\\

hence, the second claim is shown. Using Remarks \ref{evaluate}, for any given $f\in L^{1}([0,1))$, standard $\epsilon>0$, $\eta\in{^{*}\mathcal{N}}$ infinite prime;\\

$\int_{0}^{1}fd\mu-\epsilon<{1\over \eta}\sum_{i=1}^{\eta}f(\sum_{l=0}^{d}{[a_{l}\eta]\over \eta}i^{l}\ (mod\ 1))<\int_{0}^{1}fd\mu+\epsilon$\\

Again, by underflow, there exists a standard $M(\epsilon,p,f)$ prime, such that, for all standard primes $q\geq M(\epsilon,p,f)$;\\

$\int_{0}^{1}fd\mu-\epsilon<{1\over q}\sum_{i=1}^{q}f(\sum_{l=0}^{d}{[a_{l}q]\over q}i^{l}\ (mod\ 1))<\int_{0}^{1}fd\mu+\epsilon$\\

Hence, the final claim is shown.\\
\end{proof}

\begin{defn}
\label{measures}
If $p\in \mathcal{R}[x]$, and $p_{q}$, $q$ prime, are as in Lemma \ref{polynomials}, we define the associated measure $\mu_{p,q}={1\over q}(\delta_{p_{q}(1)}+\ldots+\delta_{p_{q}(i)}+\ldots+\delta_{p_{q}(q)})$, where $\{\delta_{p_{q}(i)}:1\leq i\leq q\}$ are point measures supported at $\{p_{q}(i):1\leq i\leq q\}$.
\end{defn}

\begin{lemma}
\label{convergence}
If $p\in \mathcal{R}[x]$, then the sequence $\{\mu_{p,q}:q\in\mathcal{N},\ q\ prime\}$ converges weakly to Lebesgue measure on $[0,1)$.
\end{lemma}

\begin{proof}
The proof follows immediately from the last part of Lemma \ref{polynomials}.
\end{proof}

\end{document}